\documentclass{amsart}

\usepackage[utf8]{inputenc}
\usepackage{amsmath, amsthm, latexsym, amssymb, fancyhdr, color, mathtools}
\usepackage[all, knot, color]{xy}
\usepackage{marginnote}
\usepackage{bbold}
\usepackage{changepage}
\usepackage{amscd}
\usepackage{bm}
\usepackage[bbgreekl]{mathbbol}
\usepackage{tikz}
\usetikzlibrary{cd, calc, knots}

\usepackage{graphicx}
\def\bx{{\bm{x}}}
\def\bX{{\bm{X}}}
\def\a{{\alpha}}
\def\b{{\beta}}

\def\e{{\epsilon}}
\def\w{{\omega}}

\def\o{{\otimes}}
\def\s{{\sigma}}

\def\vec{{\operatorname{Vec}}}

\def\SL2{{\mbox{SL}_2(\BZ)}}
\def\Mp2{{\mbox{Mp}_2(\BZ)}}

\def\1{{\mathbb{1}}}
\def\to{\rightarrow}

\def\BC{{\mathbb{C}}}

\def\BZ{{\mathbb{Z}}}

\def\CC{\mathcal{C}}

\def\DD{\mathcal{D}}

\def\lcm{{\operatorname{lcm}}}

\newcommand\Irr{{\operatorname{Irr}}}

\def\id{{\operatorname{id}}}
\def\lcm{{\operatorname{lcm}}}

\def\Tr{{\operatorname{Tr}}}

\def\ord{{\operatorname{ord}\,}}

\def\ev{{\operatorname{ev}}}
\def\coev{{\operatorname{coev}}}
\newcommand{\red}[1]{{\color{red} #1}}

\newcommand\ol[1]{\overline{#1}}

\newtheorem{thm}{Theorem}

\newtheorem{prop}[thm]{Proposition}

\theoremstyle{definition}

\newtheorem{remark}[thm]{Remark}

\usetikzlibrary{decorations.markings}
\tikzset{->-/.style={decoration={
			markings,
			mark=at position #1 with {\arrow{>}}},postaction={decorate}}}
\tikzset{-<-/.style={decoration={
			markings,
			mark=at position #1 with {\arrow{<}}},postaction={decorate}}}

\begin{document}

\title{Recovering $R$-symbols from modular data}
\author{Siu-Hung Ng}
\address{Department of Mathematics,
Louisiana State University,
Baton Rouge, LA 70803, USA}
%\email[Siu-Hung~Ng]{rng@@math.lsu.edu}
\author{Eric C. Rowell}
\address{ Department of Mathematics, Texas A\&M University, College
Station, TX 77843, USA }
%\email[Eric~C.~Rowell]{rowell@math.tamu.edu}
\author{Xiao-Gang Wen}
\address{Department of Physics, Massachusetts Institute of Technology,
Cambridge, Massachusetts 02139, USA}
%\email[Xiao-Gang~Wen]{xgwen@mit.edu}

\begin{abstract}
 Given a premodular category $\CC$, we show that its $R$-symbols can be recovered from its $T$-matrix,  fusion coefficients and some 2nd generalized Frobenius-Schur indicators. In particular, if $\CC$ is modular, its $R$-symbols for a certain gauge choice are explicitly determined by its modular data.  Thus, $R$-symbols cannot distinguish distinct topological orders with the same modular data.  Our approach suggests that in the multiplicity-free case the modular data determines $\CC$. \end{abstract}

\maketitle

\noindent
\textbf{Introduction}:
Modular tensor categories (MTCs) are spherical nondegenerate braided fusion
categories. These categories naturally arise in both mathematics and physics
such that  the theory of subfactors, rational conformal theory \cite{MS8977},
topological quantum field theory \cite{W8951}, topological orders
\cite{W8987,W150605768}, and many more. Recently, MTC theory becomes more
important due to a realization that symmetries in 1-dimensional quantum systems
are actually described by MTCs in trivial Witt class
\cite{KZ150201690,TW191202817,KZ200514178}, rather than by groups which can
only describe a subset of symmetries.

Associated to an MTC $\CC$ is a pair $(S,T)$ of complex square matrices, called the \emph{modular data} of $\CC$, which are  indexed by a complete set $\Irr(\CC)$ of simple objects of $\CC$, and uniquely determined by $\CC$ up to some permutations on $\Irr(\CC)$. The number $|\Irr(\CC)|$ is called the \emph{rank} of the modular data $(S,T)$.

There have been numerous efforts to generate modular data of small ranks (cf. \cite{Wen1, BNRW, NRWW, NRW}). However, it has been shown by examples that modular data of an MTC does not determine the underlying MTC  \cite{MS1}. More precisely, the nonabelian group $G$ of order $55$ admits two non-cohomologous 3-cocycles $\w$ and $\w'$ such that the centers $Z(\vec_G^\w)$ and  $Z(\vec_G^{\w'})$ are inequivalent braided fusion categories but their modular data are the same, up to a permutation. 

To completely determine an MTC $\CC$, up to equivalence, one may need its associativity constraints and braidings. It is natural to ask whether there is any invariant of MTCs beyond modular data without knowing the complete information of its associativity constraints and braiding. The Whitehead  \cite{BPD} and the Borromean links \cite{KMS} colored by simple objects of an MTC were shown to be invariants beyond modular data, and they can distinguish $Z(\vec_G^{\w})$ for any 3-cocycle $\w$ of the nonabelian group $G$ above.

These invariants motivate us to ask if the gauge equivalence classes of $R$-symbols of an MTC are invariants beyond modular data. However, we find that, for some \emph{gauge choice}, the $R$-symbols can be determined by the modular data of any MTC.   We conjecture that in the multiplicity-free case, modular data \emph{does} determine the MTC uniquely.

In general, $R$-symbols of a premodular category $\CC$ can be determined by its $T$-matrix, fusion rules and the second generalized Frobenius-Schur indicators.  When $\CC$ is modular, its fusion rules and generalized Frobenius-Schur indicators can be expressed in terms of its modular data.

The article assumes some basic knowledge on modular tensor categories, graphical calculus and Frobenius-Schur indicators for pivotal categories. The readers are referred to \cite{EGNO, BaKi, KasBook} for more details. The categorical development of Frobenius-Schur indicators can be found in \cite{NS1, NS2, NS4}.

\noindent
\textbf{$R$-symbols of premodular categories}:
%\label{s1}
Let $\CC$ be a premodular (or ribbon) category  over $\BC$, and $\Irr(\CC)$ denote a complete set of nonisomorphic simple objects containing the unit object $\1$. 

For any objects $x,y \in \CC$, the morphism space of $\CC$ from $x$ to $y$, denoted by $\CC(x,y)$, is a finite-dimensional $\BC$-linear space.  Given any $a,b,c \in \Irr(\CC)$, one can choose a basis $
B^{a,b}_c$ for the $\BC$-linear space $\CC(a \o b, c)$. The collection 
$$
B= \big\{B^{a,b}_c \mid a,b,c \in \Irr(\CC)\big\}
$$
is simply called a \emph{gauge choice} of $\CC$. In particular, we write  $N^{a,b}_c$ for $|B^{a,b}_c|$. The semisimplicity of $\CC$ immediately implies  that $N_c^{a,b}$ is the fusion coefficient determined by the tensor product
$$
a \o b = \bigoplus_{c \in \Irr(\CC)} N_c^{a,b} \, c\quad \text{for any } a, b \in \Irr(\CC)\,.
$$

For the purpose of using graphical calculus, we may assume $\CC$ is a strict pivotal category without loss of generality \cite[Thm. 2.2]{NS1}. A morphism $\pi \in \CC(a \o b, c)$ will be depicted as a trivalent tree: 
$\pi=  
\xy 0;/r1pc/:
(0,0)*{\bullet},
(0,0); (-1,1.2)*+{a} **\dir{-},
(0,0); (1,1.2)*+{b} **\dir{-},
(0,0); (0,-1.2)*+{c} **\dir{-}
\endxy
$. 
We may label the vertex of the trivalent tree to distinguish  possibly different morphisms in $\CC(a \o b, c)$. 
Thus, we may write 
$$
B^{a,b}_c = \left\{ \xy 0;/r1pc/:
(0,0)*{\bullet},
(0,0); (-1,1.2)*+{a} **\dir{-},
(0,0); (1,1.2)*+{b} **\dir{-},
(0,0); (0,-1.2)*+{c} **\dir{-},
(0,.5)*{\scriptstyle \a}
\endxy
: \a=1,\dots,N^{a,b}_c 
\right\}\,.
$$

The braiding $R_{x,y}: x \o y \to y \o x$ of $\CC$ for any $x,y \in \CC$ is depicted by the diagram: 
$$
R_{x,y} \, = \, \vcenter{\xy 0;/r1.5pc/:
(0,1)\vtwistneg[1],
(0,1.2)*{\scriptstyle{x}},
(1,1.2)*{\scriptstyle{y}},
(0,-.3)*{\scriptstyle{y}},
(1,-.3)*{\scriptstyle{x}},
\endxy}\,.
$$
For any $x, y \in \CC$, the braiding $R_{x,y}$ acts as a linear isomorphism, 
$$
R_{x,y}^z: \CC(y \o x, z) \to \CC(x \o y, z)
$$
by composition, for any $z \in \CC$. Specifically, $R^z_{x,y}(\pi) = \pi \circ R_{x,y}$ for $\pi \in \CC(y \o x, z)$. In terms of diagrams, 
$$
R^z_{x,y}: \vcenter{\xy 0;/r1pc/:
(0,0)*{\bullet},
(0,0); (-1,1.2)*+{\scriptstyle{y}} **\dir{-},
(0,0); (1,1.2)*+{\scriptstyle{x}} **\dir{-},
(0,0); (0,-1.2)*+{\scriptstyle{z}} **\dir{-}
\endxy} \mapsto \,
\vcenter{\xy 0;/r1pc/:
\vtwistneg~{(-1,2)*+{\scriptstyle{x}}}{(0.5,2)*+{\scriptstyle{y}}}{(-1,.5)}{(0.5,.5)},
(-.25,-.3)*{\bullet},
(-.25,-.3); (-1,.5) **\crv{(-1,.2)},
(-.25,-.3); (.5,.5) **\crv{(.5, .2)},
(-.25,-.3); (-.25,-1.2)*+{\scriptstyle{z}} **\dir{-},
\endxy}\,.
$$
Since $R_{x,y}$ is an isomorphism in $\CC$, if $B^{y,x}_z$ is a basis for $\CC(y \o x, z)$ then 
$$
B^{x,y}_z=\{R_{x,y}(\pi) \mid \pi \in B^{y,x}_z\}
$$
is a basis for $\CC(x \o y, z)$. 

Let $\theta$ be the ribbon structure of $\CC$, which means  $\theta$ is a natural isomorphism of the identity functor of $\CC$ satisfying the conditions:
\begin{equation}\label{eq:ribbon}
    (\theta_x)^*  = \theta_{x^*} \quad\text{and}\quad \theta_{x \o y} = (\theta_x \o \theta_y) R_{y,x} R_{x,y}
\end{equation}
for any $x, y \in \CC$. Thus, for any $a \in \Irr(\CC)$, $\theta_a$ is a scalar multiple of $\id_a$, and we will abuse notation by denoting this scalar by $\theta_a$. In general, $\theta_x: x \to x$ is an isomorphism in $\CC$. Since $\CC$ is assumed to be strict pivotal, we can depict $\theta_x$ as
$$
\theta_x \,=\, \vcenter{\xy 0;/r1pc/:
(0,2)\vcap[1],
\vtwistneg~{(1,2)}{(2,2)}{(1, 1)}{(2,1)},
(2,2); (2,3) **\dir{-}?(1)+(0.4,0)*{\scriptstyle x},,
(2,1); (2,0) **\dir{-}?(1)+(0.4,0)*{\scriptstyle x},
(0,2); (0,1) **\dir{-},
(0,1)\vcap[-1],\endxy}\,,
$$
where the coevaluation $\coev: \1 \to y \o y^*$ and the evaluation $\ev: y^* \o y \to \1$ morphisms are respectively depicted by the diagrams  
$$
\vcenter{\xy 0;/r1pc/:
(0,2.1)\vcap[2],
(2.2,1.8)*{\scriptstyle y^*},
(0,1.8)*{\scriptstyle y},
(5,2)*{,},
(8,2.7)\vcap[-2],
(8,3.3)*{\scriptstyle y^*},
(10,3.3)*{\scriptstyle y},
(12,2)*{.},
\endxy}
$$
In particular, for $a, b \in \Irr(\CC)$,
$$
\vcenter{\xy 0;/r1pc/:
(0,2)\vcap[1],
\vtwistneg~{(1,2)}{(2,2)}{(1, 1)}{(2,1)},
(2,2); (2,3) **\dir{-}?(1)+(0.4,0)*{\scriptstyle a},,
(2,1); (2,0) **\dir{-}?(1)+(0.4,0)*{\scriptstyle a},
(0,2); (0,1) **\dir{-},
(0,1)\vcap[-1],\endxy} = \theta_a\, \,
\vcenter{\xy 0;/r1pc/:
(2,3); (2,0) **\dir{-}?(.5)+(0.4,0)*{\scriptstyle a},\endxy} \quad\text{ and so }
(\theta_a)^* = 
\vcenter{\xy 0;/r1pc/:
(0,2)\vcap[1],
\vtwistneg~{(1,2)}{(2,2)}{(1, 1)}{(2,1)},
(2,2); (2,3) **\dir{-}?(.8)+(-0.4,0)*{\scriptstyle a},
(2,3)\vcap[1],
(2,1); (2,0) **\dir{-},
(0,2); (0,1) **\dir{-},
(-.5,0)\vcap[-2.5],
(0,1)\vcap[-1],
(-0.5,0); (-.5,4) **\dir{-}?(1)+(-0.4,0)*{\scriptstyle a^*},
(3,3); (3,-2) **\dir{-}?(1)+(0.4,0)*{\scriptstyle a^*},
\endxy} = \theta_a\, \,
\vcenter{\xy 0;/r1pc/:
(2,3); (2,0) **\dir{-}?(.5)+(0.4,0)*{\scriptstyle a^*},\endxy}
$$
and the second part of \eqref{eq:ribbon} becomes
\begin{equation} \label{eq:ribbon2}
 \vcenter{\xy 0;/r1.5pc/:
\vtwistneg~{(-1,3.5)*+{\scriptstyle{a}}}{(0.5,3.5)*+{\scriptstyle{b}}}{(-1,2)}{(0.5,2)},
\vtwistneg~{(-1,2)}{(0.5,2)}{(-1,.5)*+{\scriptstyle{a}}}{(0.5,.5)*+{\scriptstyle{b}}},
\endxy} =
\theta_a^{-1}\theta_b^{-1}\, \vcenter{\xy 0;/r1pc/:
\vcross~{(0,1)}{(.5,1)}{(0,0)}{(.5, 0)},
\vcross~{(.5,0)}{(1.5,0)}{(.5,-1)}{(1.5,-1.5)},
\vcross~{(0,-1)}{(.5,-1)}{(0,-2)}{(.5,-2)},
\vcross~{(-.5,0)}{(0,0)}{(-.5,-1)}{(0,-1)},
(-.5,-1);(-.5,-2) **\dir{-},
(-.5,0);(-.5,1) **\dir{-},
(-1.5,1)\vcap[1],
(-2,1)\vcap[2],
(.5,1);(.5,2)*+{\scriptstyle a} **\dir{-},
(1.5,0);(1.5,2)*+{\scriptstyle b} **\dir{-},
(-1.5,-2)\vcap[-1],
(.5,-2);(.5,-3)*+{\scriptstyle a} **\dir{-},
(-1.5,1);(-1.5,-2)**\dir{-},
(-2,-2)\vcap[-2],
(-2,1);(-2,-2) **\dir{-},
(1.5,-1.5);(1.5,-3)*+{\scriptstyle{b}} **\dir{-},
\endxy}\,.
\end{equation}

\begin{prop} \label{p1}
Let $\CC$ be a premodular category over $\BC$ with the ribbon structure $\theta$ and braiding $R$. Assume an arbitrary complete order on $\Irr(\CC)$.  There exists a gauge choice $B$ of $\CC$ such that the matrix $[R_{a,b}^c]$ relative to the corresponding bases  of $\CC(b \o a, c)$ and $\CC(a \o b, c)$ is a diagonal matrix for any $a,b,c \in \Irr(\CC)$. Moreover,
$$
[R_{a,b}^c] = \left\{\begin{array}{ll}
\id & \text{if } a > b,\\
\frac{\theta_c}{\theta_a \theta_b} \id & \text{if } a < b,\\
\frac{\sqrt{\theta_c}}{\theta_a} \Delta_a^c & \text{if } a = b,
\end{array}\right.
$$
where $\theta_c$ is the scalar of the ribbon structure specified at $c \in \Irr(\CC)$, $\sqrt{\theta_c}$ is any chosen square root of $\theta_c$ for each $c \in \Irr(\CC)$, and $\Delta_a^c$ is a signed diagonal matrix.
\end{prop}
\begin{proof}
For $a, b, c \in \Irr(\CC)$ with $a < b$, we (arbitrarily) fix a basis $B^{a,b}_c$ for $\CC(a\o b, c)$. If $a > b$, then $B_c^{a,b}:=R_{a,b}^c( B^{b,a}_c)$ is a basis for $\CC(a\o b, c)$ since $R^c_{a,b}$ is a $\BC$-linear isomorphism. Thus, if $a > b$, the matrix $[R^c_{a,b}]$ of  $R^c_{a,b}$ relative to these bases is the identity matrix $\id$. On the other hand, if $a < b$, then 
$$
B_c^{b, a}=\left\{
\vcenter{\xy 0;/r1pc/:
\vtwistneg~{(-1,2)*+{\scriptstyle{b}}}{(0.5,2)*+{\scriptstyle{a}}}{(-1,.5)}{(0.5,.5)},
(-.25,-.3)*{\bullet},
(-.25,-.3); (-1,.5) **\crv{(-1,.2)},
(-.25,-.3); (.5,.5) **\crv{(.5, .2)},
(-.25,-.3); (-.25,-1.2)*+{\scriptstyle{c}} **\dir{-},
(-.25,.1)*{\scriptstyle{\a}},
\endxy}: 
\vcenter{\xy 0;/r1pc/:
(0,0)*{\bullet},
(0,0); (-1,1.2)*+{\scriptstyle{a}} **\dir{-},
(0,0); (1,1.2)*+{\scriptstyle{b}} **\dir{-},
(0,0); (0,-1.2)*+{\scriptstyle{c}} **\dir{-},
(0,.4)*{\scriptstyle{\a}},
\endxy} \in B_c^{a,b}\right\}
$$
and 
\begin{align}
&R_{a,b}^c\left(
\vcenter{\xy 0;/r1pc/:
\vtwistneg~{(-1,2.5)*+{\scriptstyle{b}}}{(0.5,2.5)*+{\scriptstyle{a}}}{(-1,.5)}{(0.5,.5)},
(-.25,-.3)*{\bullet},
(-.25,-.3); (-1,.5) **\crv{(-1,.2)},
(-.25,-.3); (.5,.5) **\crv{(.5, .2)},
(-.25,-.3); (-.25,-1.2)*+{\scriptstyle{c}} **\dir{-},
(-.25,.1)*{\scriptstyle{\a}},
\endxy}\right)
 = \vcenter{\xy 0;/r1pc/:
\vtwistneg~{(-1,3.5)*+{\scriptstyle{a}}}{(0.5,3.5)*+{\scriptstyle{b}}}{(-1,2)}{(0.5,2)},
\vtwistneg~{(-1,2)}{(0.5,2)}{(-1,.5)}{(0.5,.5)},
(-.25,-.3)*{\bullet},
(-.25,-.3); (-1,.5) **\crv{(-1,.2)},
(-.25,-.3); (.5,.5) **\crv{(.5, .2)},
(-.25,-.3); (-.25,-1.2)*+{\scriptstyle{c}} **\dir{-},
(-.25,.1)*{\scriptstyle{\a}},
\endxy} =
\theta_a^{-1}\theta_b^{-1} \vcenter{\xy 0;/r1pc/:
\vcross~{(0,1)}{(.5,1)}{(0,0)}{(.5, 0)},
\vcross~{(.5,0)}{(1.5,0)}{(.5,-1)}{(1.5,-1.5)},
\vcross~{(0,-1)}{(.5,-1)}{(0,-2)}{(.5,-2)},
\vcross~{(-.5,0)}{(0,0)}{(-.5,-1)}{(0,-1)},
(-.5,-1);(-.5,-2) **\dir{-},
(-.5,0);(-.5,1) **\dir{-},
(-1.5,1)\vcap[1],
(-2,1)\vcap[2],
(.5,1);(.5,2)*+{\scriptstyle a} **\dir{-},
(1.5,0);(1.5,2)*+{\scriptstyle b} **\dir{-},
(-1.5,-2)\vcap[-1],
(.5,-2);(.5,-3) **\dir{-},
(-1.5,1);(-1.5,-2)**\dir{-},
(-2,-2)\vcap[-2],
(-2,1);(-2,-2) **\dir{-},
(1.5,-1.5);(1.5,-3) **\dir{-},
(1,-3.5)*{\bullet},
(1,-3.5); (.5,-3) **\crv{(.5,-3.2)},
(1,-3.5); (1.5,-3) **\crv{(1.5, -3.2)},
(1,-3.5);(1,-4.3)*+{\scriptstyle c} **\dir{-},
(1,-3)*{\scriptstyle{\a}},
\endxy} 
\nonumber\\
& =
\theta_a^{-1}\theta_b^{-1} \, \vcenter{\xy 0;/r1pc/:
(0,2)\vcap[1],
\vtwistneg~{(1,2)}{(2,2)}{(1, 1)}{(2,1)},
(2,2); (2,3) **\dir{-},
(2,1); (2,-.5) **\dir{-}?(1)+(0.4,0)*{c},
(0,2); (0,1) **\dir{-},
(0,1)\vcap[-1],
(2,3)*{\bullet},(2,3.5)*{\scriptstyle{\a}},
(2,3);(1,4.5)*+{\scriptstyle a} **\dir{-},
(2,3);(3,4.5)*+{\scriptstyle b} **\dir{-},
\endxy}
=
\frac{\theta_c}{\theta_a\theta_b} \vcenter{\xy 0;/r1pc/:
(0,0)*{\bullet},
(0,0); (-1,1.2)*+{\scriptstyle{a}} **\dir{-},
(0,0); (1,1.2)*+{\scriptstyle{b}} **\dir{-},
(0,0); (0,-1.2)*+{\scriptstyle{c}} **\dir{-},
(0,.5)*{\scriptstyle{\a}}
\endxy},
\end{align}
where the second equality follows from \eqref{eq:ribbon2} and third equality is a consequence of the naturality of the ribbon structure. 
Thus, the matrix $[R_{a,b}^c]$ of $R_{a,b}^c$ relative to the chosen bases is $\frac{\theta_c}{\theta_a \theta_b} \id$.

Following from \eqref{eq:ribbon}, $(\theta_a \o \theta_a)\circ (R_{a,a})^2 = \theta_{a\o a}$ for all $a \in \Irr(\CC)$. So $(R_{a,a})^{2N}= \id_{a \o a}$ for all $a \in \Irr(\CC)$, where $N = \ord(\theta) = \mathop{\lcm}\limits_{c \in \Irr(\CC)} \ord(\theta_a)$. Thus, $R_{a,a}^c$ is a finite order operator on $\CC(a\o a,c)$ for any $c \in \Irr(\CC)$. In particular, $R_{a,a}^c$ is diagonalizable. Therefore, there exists a basis $B^{a,a}_c$ of $\CC(a \o a, c)$  consisting of eigenvectors of $R_{a,a}^c$. 

Let $\pi_\a = \vcenter{\xy 0;/r1pc/:
(0,0)*{\bullet},
(0,0); (-1,1.2)*+{\scriptstyle{a}} **\dir{-},
(0,0); (1,1.2)*+{\scriptstyle{a}} **\dir{-},
(0,0); (0,-1.2)*+{\scriptstyle{c}} **\dir{-},
(0,.5)*{\scriptstyle{\a}},
\endxy} \in B^{a,a}_c$ such that $R_{a,a}^c\left(
 \vcenter{\xy 0;/r1pc/:
(0,0)*{\bullet},
(0,0); (-1,1.2)*+{\scriptstyle{a}} **\dir{-},
(0,0); (1,1.2)*+{\scriptstyle{a}} **\dir{-},
(0,0); (0,-1.2)*+{\scriptstyle{c}} **\dir{-},
(0,.5)*{\scriptstyle{\a}},
\endxy}
\right) = \lambda_\a \,
 \vcenter{\xy 0;/r1pc/:
(0,0)*{\bullet},
(0,0); (-1,1.2)*+{\scriptstyle{a}} **\dir{-},
(0,0); (1,1.2)*+{\scriptstyle{a}} **\dir{-},
(0,0); (0,-1.2)*+{\scriptstyle{c}} **\dir{-},
(0,.5)*{\scriptstyle{\a}},
\endxy}$ 
for some $\lambda_\a \in \BC^\times$. Then, by the same calculation as above, we find
\begin{align}
&\lambda_\a^2   \vcenter{\xy 0;/r1pc/:
(0,0)*{\bullet},
(0,0); (-1,1.2)*+{\scriptstyle{a}} **\dir{-},
(0,0); (1,1.2)*+{\scriptstyle{a}} **\dir{-},
(0,0); (0,-1.2)*+{\scriptstyle{c}} **\dir{-},
(0,.5)*{\scriptstyle{\a}},
\endxy} = (R_{a,a}^c)^2\left(
 \vcenter{\xy 0;/r1pc/:
(0,0)*{\bullet},
(0,0); (-1,1.2)*+{\scriptstyle{a}} **\dir{-},
(0,0); (1,1.2)*+{\scriptstyle{a}} **\dir{-},
(0,0); (0,-1.2)*+{\scriptstyle{c}} **\dir{-},
(0,.5)*{\scriptstyle{\a}},
\endxy}
\right)=\vcenter{\xy 0;/r1pc/:
\vtwistneg~{(-1,3.5)*+{\scriptstyle{a}}}{(0.5,3.5)*+{\scriptstyle{a}}}{(-1,2)}{(0.5,2)},
\vtwistneg~{(-1,2)}{(0.5,2)}{(-1,.5)}{(0.5,.5)},
(-.25,-.3)*{\bullet},
(-.25,-.3); (-1,.5) **\crv{(-1,.2)},
(-.25,-.3); (.5,.5) **\crv{(.5, .2)},
(-.25,-.3); (-.25,-1.2)*+{\scriptstyle{c}} **\dir{-},
(-.25,.1)*{\scriptstyle{\a}},
\endxy} 
\nonumber\\
&=
\theta_a^{-2} \vcenter{\xy 0;/r1pc/:
\vcross~{(0,1)}{(.5,1)}{(0,0)}{(.5, 0)},
\vcross~{(.5,0)}{(1.5,0)}{(.5,-1)}{(1.5,-1.5)},
\vcross~{(0,-1)}{(.5,-1)}{(0,-2)}{(.5,-2)},
\vcross~{(-.5,0)}{(0,0)}{(-.5,-1)}{(0,-1)},
(-.5,-1);(-.5,-2) **\dir{-},
(-.5,0);(-.5,1) **\dir{-},
(-1.5,1)\vcap[1],
(-2,1)\vcap[2],
(.5,1);(.5,2)*+{\scriptstyle a} **\dir{-},
(1.5,0);(1.5,2)*+{\scriptstyle a} **\dir{-},
(-1.5,-2)\vcap[-1],
(.5,-2);(.5,-3) **\dir{-},
(-1.5,1);(-1.5,-2)**\dir{-},
(-2,-2)\vcap[-2],
(-2,1);(-2,-2) **\dir{-},
(1.5,-1.5);(1.5,-3) **\dir{-},
(1,-3.5)*{\bullet},
(1,-3.5); (.5,-3) **\crv{(.5,-3.2)},
(1,-3.5); (1.5,-3) **\crv{(1.5, -3.2)},
(1,-3.5);(1,-4.3)*+{\scriptstyle c} **\dir{-},
(1,-3)*{\scriptstyle{\a}},
\endxy}=
\frac{\theta_c}{\theta_a^2} \,\vcenter{\xy 0;/r1pc/:
(0,0)*{\bullet},
(0,0); (-1,1.2)*+{\scriptstyle{a}} **\dir{-},
(0,0); (1,1.2)*+{\scriptstyle{b}} **\dir{-},
(0,0); (0,-1.2)*+{\scriptstyle{c}} **\dir{-},
(0,.5)*{\scriptstyle{\a}},
\endxy}\,.
\end{align}
Therefore, $\lambda_\a = \e_\a \frac{\sqrt{\theta_c}}{\theta_a}$ where $\e_\a = \pm 1$, and the matrix $[R_{a,a}^c]$ relative to the basis $B^{a,a}_c$ is of the form $\frac{\sqrt{\theta_c}}{\theta_a} \Delta_a^c$ where $\Delta_a^c$ is the matrix given by $(\Delta_a^c)_{\a, \a'} = \e_\a \delta_{\a, \a'}$.
\end{proof}
\begin{remark}\label{r1}
For any $a, c \in \Irr(\CC)$, the last proposition show that the eigenvalues of $R_{a,a}^c$ can only be $\pm \frac{\sqrt{\theta_c}}{\theta_a}$. Let $d_{\pm}$ be the multiplicity of the eigenvalue $\pm \frac{\sqrt{\theta_c}}{\theta_a}$ of $R_{a,a}^c$ respectively.  One can reorder the basis $B^{a,a}_c$ so that the matrix $[R_{a,a}^c]$ relative to the basis $B_c^{a,a}$ is given by
$[R_{a,a}^c]=   \frac{\sqrt{\theta_c}}{\theta_a}\Delta_a^c$ with
$$
\Delta_a^c = \left[\begin{array}{c|c} I_{d_+}& 0\\\hline 0& -I_{d_-} \end{array}\right]\,,
$$
where $I_d$ denotes the identity matrix of rank $d$
\end{remark}
The $S$ and $T$-matrices of a premodular category are insufficient to determine the matrix $[R_{a,a}^c]$ for the gauge choice in the preceding proposition as the multiplicity of 1 (or $-1$) in the diagonal of $\Delta_{a,c}$ cannot be determined.  However, we will prove that the $T$-matrix, the fusion coefficient $N^{a,a}_c$ and the \emph{generalized Frobenius-Schur indicator} $\nu^{\iota(c)}_{2,1}(a)$  determine $\Delta_a^c$, and hence $R_{a,a}^c$, up to a permutation of $B^{a,a}_c$, for all $a, c\in \Irr(\CC)$, where $\iota: \CC \to Z(\CC)$ is an embedding.  In particular, if $\CC$ is an MTC, $\nu^{\iota(c)}_{2,1}(a)$ and $N^{a,a}_c$ can be expressed in terms of the $S$- and $T$-matrices.

\noindent
\textbf{The main result}:
We continue to use the conventions developed in this paper,
%in Section \ref{s1}, 
and consider a premodular subcategory  $\CC$ as a (premodular) subcategory of its center $Z(\CC)$. 

Recall that the objects in the center $Z(\CC)$ of $\CC$ are pairs $(x, \s_{x,-})$ in which $x \in \CC$ and $\s_{x,y}: x \o y \to y \o x$ is a half-braiding for any object $y$ of $\CC$ \cite{KasBook}. In particular, $\s_{x,y}$ is a natural isomorphism in $y$ for all $y \in \CC$. We will depict the half-braiding $\s_{x,y}$  of $(x, \s_{x,-})$  for any $y \in \CC$ as
$$
\s_{x,y} = \vcenter{\xy 0;/r.5pc/:
{\ar@[red]@{-}@`{(0,-2.5), (3,-3)} (0,0)*+{\scriptstyle \red{x}}; (3,-5)},
{\ar@{-}@`{(3,-2.5),(0,-3)}|(0.5){\hole} (3,0)*+{\scriptstyle y}; (0,-5)},
\endxy}\,.
$$
Note that the object $x$ is colored in red as the half-braiding $\s_{x, y}$ is only natural in $y$.

For $x \in \CC$, $\bx =(x, R^{-1}_{-,x}) \in Z(\CC)$. The assignment $x \mapsto \bx$ can be extended to a fully faithful braided tensor functor (or simply embedding) $\iota : \CC \to Z(\CC)$ which preserves the spherical structures  and hence the dimensions. So we can consider  $\CC$ as a premodular subcategory of $Z(\CC)$ under the identification $x \mapsto \bx = (x, R^{-1}_{-,x}) \in Z(\CC)$.  

Now we recall the definition \cite[\S 2]{NS4} of generalized Frobenius-Schur indicators $\nu_{2,1}^\bX(z)$ for any $\bX=(x, \s_{x,-}) \in Z(\DD)$ and $z \in \DD$, where $\DD$ is any \emph{strict spherical fusion category}. Consider the linear operator $E^{(2,1)}_{\bX, z}: \DD(x, z\o z) \to \DD(x, z\o z)$ given by
$$
E^{(2,1)}_{\bX, z}: \vcenter{\xy 0;/r1pc/:
(0,0)*{\bullet},
(0,0); (-1,-1.2)*+{\scriptstyle{z}} **\dir{-},
(0,0); (1,-1.2)*+{\scriptstyle{z}} **\dir{-},
{\ar@{-} (0,0);(0,1.2)*+{\scriptstyle x}},
\endxy} \mapsto \,\,
\vcenter{\xy 0;/r1pc/:
(0,0)*{\bullet},
{\ar@{-}@`{(-1, -1.2), (-2, 0),(0.2,1.7), (1,1.2)}|(0.55)\hole (0,0);(2.5,-1.2)*+{\scriptstyle z}},
{\ar@[red]@{-}@`{(0,1),(-1,.7)}(0,0);(-1,2)*+{\scriptstyle \red{x}}},
(0,0); (1,-1.2)*+{\scriptstyle{z}} **\dir{-},
\endxy}
$$
The superscript $(2,1)$ means there are 2 legs labeled by $z$, and the fisrt one is rotated to the other side of the second leg labelled by $z$ via the half-braiding $\s_{x,z}$.
The generalized Frobenius-Schur indicators $\nu_{2,1}^\bX(z)$ is defined as
$
\nu_{2,1}^\bX(z) = \Tr(E^{(2,1)}_{\bX, z})\,.
$
The readers are referred to  \cite[\S 2]{NS4} for more details on generalized Frobenius-Schur indicators, their properties and some applications.

Now we return to the embedding $\iota: \CC \to Z(\CC)$ described above for the premodular category $\CC$. In this case, for $c, a \in \Irr(\CC)$,  $\iota(c) = (c,\s_{c,-})$ with the half-braiding $\s_{c,a} = R_{a,c}^{-1}$. In terms of diagrams, we have
$
\vcenter{\xy 0;/r.4pc/:
{\ar@[red]@{-}@`{(0,-2.5), (3,-3)} (0,0)*+{\scriptstyle \red{c}}; (3,-5)},
{\ar@{-}@`{(3,-2.5),(0,-3)}|(0.5){\hole} (3,0)*+{\scriptstyle a}; (0,-5)},
\endxy} = \vcenter{\xy 0;/r.4pc/:
{\ar@{-}@`{(0,-2.5), (3,-3)}|(0.5){\hole} (0,0)*+{\scriptstyle c}; (3,-5)},
{\ar@{-}@`{(3,-2.5),(0,-3)} (3,0)*+{\scriptstyle a}; (0,-5)},
\endxy}\,. 
$  
Thus,
$$
E^{(2,1)}_{\iota(c), a}: \vcenter{\xy 0;/r1pc/:
(0,0)*{\bullet},
(0,0); (-1,-1.2)*+{\scriptstyle{a}} **\dir{-},
(0,0); (1,-1.2)*+{\scriptstyle{a}} **\dir{-},
{\ar@{-} (0,0);(0,1.2)*+{\scriptstyle c}},
\endxy} \mapsto \,\,
\vcenter{\xy 0;/r1pc/:
(0,0)*{\bullet},
{\ar@{-}@`{(-1, -1.2), (-2, 0),(0.2,1.7), (1,1.2)} (0,0);(2.5,-1.2)*+{\scriptstyle a}},
{\ar@{-}@`{(0,1),(-1,.7)}|(.55)\hole (0,0);(-1,2)*+{\scriptstyle c}},
(0,0); (1,-1.2)*+{\scriptstyle{a}} **\dir{-},
\endxy}
=\,
\vcenter{\xy 0;/r1pc/:
(0,0)*{\bullet},
{\ar@{-}@`{ (-1, -2),(-2, -1.5), (-1, .3)}|(.08)\hole (0,0);(1.5,-2.5)*+{\scriptstyle a}},
{\ar@{-}(0,0);(0,1.5)*+{\scriptstyle c}},
{\ar@{-}@`{ (2, -.4)}|(.72)\hole (0,0);(-.7,-2.5)*+{\scriptstyle a}},
\endxy}
=\,\theta_a\,
\vcenter{\xy 0;/r1pc/:
(0,0)*{\bullet},
{\ar@{-}(0,0);(0,1.5)*+{\scriptstyle c}},
{\ar@{-}@`{(.7,-.4)} (0,0);(.7,-1)},
{\ar@{-}@`{(-.7,-.4)} (0,0);(-.7,-1)},
\vtwistneg~{(-.7,-1)}{(.7,-1)}{(-.7,-2.5)}{(.7,-2.5)},
\endxy}\,.
$$
 Note that the semisimplicity of $\CC$ implies that the compositions of morphisms define a nondegenerate pairing 
 $$
 \CC(a \o b, c)\mathop{\o}\limits_\BC\CC(c, a \o b) \to \CC(c,c) =\BC
 $$
 for any $a,b, c \in \Irr(\CC)$. To compute the trace of $ E^{(2,1)}_{\iota(c), a}$, one can pick any gauge choice $B=\{B_c^{a,b} \mid a,b,c \in \Irr(\CC)\}$ of $\CC$. The pairing determines a dual basis $B^c_{a,b}$ for $\CC(c, a \o b)$ such that  
 $$
 \vcenter{\xy 0;/r1pc/:
(0,0)*{\bullet},
(0,-2)*{\bullet},
{\ar@{-}(0,0);(0,1.5)*+{\scriptstyle c}},
{\ar@{-}@`{(1,-.4), (1,-1.6)} (0,0);(0,-2)},
{\ar@{-}@`{(-1,-.4), (-1,-1.6)} (0,0);(0,-2)},
{\ar@{-}(0,-2);(0,-3.5)*+{\scriptstyle c}},
(-1.1, -1)*+{\scriptstyle a}, 
(1.1, -1)*+{\scriptstyle b}, 
(0, -1.6)*+{\scriptstyle \a}, 
(0, -.5)*+{\scriptstyle \b}, 
\endxy}=\,\delta_{\a, \b}
\vcenter{\xy 0;/r1pc/: 
{\ar@{-}(0,1.5)*+{\scriptstyle c};(0,-3.5)*+{\scriptstyle c}},
\endxy}\quad \text{ for}\quad \a, \b =1, \dots, N^{a,b}_c\,.
$$
Therefore, for any $a,c \in \Irr(\CC)$, 
\begin{equation} \label{eq:2nd_indicator_R}
\nu_{2,1}^{\iota(c)}(a) =\Tr(E^{(2,1)}_{\iota(c),a)})=\theta_a \sum_{\a =1}^{N^{a,a}_c} 
\vcenter{\xy 0;/r1pc/:
(0,0)*{\bullet},
{\ar@{-}(0,0);(0,1.5)*+{\scriptstyle c}},
{\ar@{-}@`{(.7,-.4)} (0,0);(.7,-1)},
{\ar@{-}@`{(-.7,-.4)} (0,0);(-.7,-1)},
\vtwistneg~{(-.7,-1)}{(.7,-1)}{(-.7,-2.5)}{(.7,-2.5)},
{\ar@{-}@`{(.7,-3.1)} (0,-3.5);(.7,-2.5)},
{\ar@{-}@`{(-.7,-3.1)} (0,-3.5);(-.7,-2.5)},
(0,-3.5)*{\bullet},
(0, -3.1)*+{\scriptstyle \a}, 
(0, -.4)*+{\scriptstyle \a}, 
{\ar@{-}(0,-3.5);(0,-5)*+{\scriptstyle c}},
\endxy} =\theta_a \Tr(R_{a,a}^c)\,. 
 \end{equation}
Hence $\Tr(R_{a,a}^c) =  \theta_a^{-1} \nu_{2,1}^{\iota(c)}(a)$.

It follows from Proposition \ref{p1} and the subsequent Remark \ref{r1} that $R_{a,a}^c$ can only have the eigenvalues $\frac{\pm \sqrt{\theta_c}}{\theta_a}$, and the matrix $[R_{a,a}^c]$ relative to some basis $B^{a,a}_c$ for $\CC(a \o a, c)$ is given by
$$
[R_{a,a}^c] =  \frac{\sqrt{\theta_c}}{\theta_a}\left[\begin{array}{c|c} I_{d_+}& 0\\\hline 0& -I_{d_-} \end{array}\right]
$$
where $d_\pm$ are the multiplicities of the eigenvalues $\frac{\pm \sqrt{\theta_c}}{\theta_a}$ respectively. Note that $d_++d_-= N^{a,a}_c$. By \eqref{eq:2nd_indicator_R},
$$
d_+- d_- = \frac{\theta_a}{\sqrt{\theta_c}} \Tr(R_{a,a}^c) =\frac{1}{\sqrt{\theta_c}} \nu_{2,1}^{\iota(c)}(a)
$$
and so
$
0\le d_{\pm}= \frac{1}{2}(N^{a,a}_c \pm \frac{1}{\sqrt{\theta_c}} \nu_{2,1}^{\iota(c)}(a)).
$
Thus,  $\frac{1}{\sqrt{\theta_c}} \nu_{2,1}^{\iota(c)}(a)$ is an integer which has the same parity as $N_c^{a,a}$ (i.e. they are congruent modulo 2).

The value of $\nu_{2,1}^{\iota(c)}(a)$ is completely determined by the modular data of the center $Z(\CC)$, and the forgetful functor $Z(\CC) \to \CC$ \cite[Cor.  5.6]{NS4}.  Precisely,  
\begin{equation}\label{eq:gen_2nd_indicator}
\nu_{2,1}^{\iota(c)}(a) = \frac{1}{\dim(\CC)} \sum_{\bX \in \Irr(Z(\CC))} [x: a]_\CC\,  \mathbb{S}_{\iota(c), \bX}  \, \bbtheta_\bX^2
\end{equation}
where $[x: a]_\CC = \dim \CC(x, a)$  if $\bX=(x, \sigma_{x,-})$, $\bbtheta_\bX$ and $\mathbb{S}_{\iota(c), \bX}$ are the modular data of $Z(\CC)$.    Since the value of $\Tr(R_{a,a}^c)$ and the fusion coefficient $N_c^{a,a}$ completely determine the multiplicities of the eigenvalues $\frac{\sqrt{\theta_c}}{\theta_a}$ of the matrix $[R_{a,a}^c]$ presented in Proposition \ref{p1} up to an ordering of the gauge choice. This proves the first part of following statement.
\begin{thm}\label{t:2}
Let $\CC$ be a premodular category over $\BC$. The braiding of $\CC$ is completely determined by its fusion rules, its ribbon structure and the generalized Frobenius-Schur indicators $\nu_{2,1}^{\iota(c)}(a)$ for all $a, c \in \Irr(\CC)$.  The value $\frac{1}{\sqrt{\theta_c}}\nu_{2,1}^{\iota(c)}(a)$ is an integer bounded by $N_c^{a,a}$ and has the same parity  as $N_c^{a,a}$ for any $a, c \in \Irr(\CC)$. 

The matrix $[R_{a,a}^c]$ in Proposition \ref{p1} is $\frac{\sqrt{\theta_c}}{\theta_a}\left[\begin{array}{c|c} I_{d_+}& 0\\\hline 0& I_{d_-} \end{array}\right]$ for some basis $B^{a,a}_c$ of $\CC(a\o a, c)$, and 
\begin{align}\label{eq:multipicity}
d_{\pm}= \frac{1}{2}(N^{a,a}_c \pm \frac{1}{\sqrt{\theta_c}} \nu_{2,1}^{\iota(c)}(a))\,.
\end{align}

In particular, if $\CC$ is modular, the braiding of $\CC$ is completely determined by the modular data $(S,T)$ of $\CC$ and
\begin{equation} \label{eq2}
    d_{\pm}
 =  \frac{1}{2}\left(N^{a,a}_c \pm \frac{1}{\sqrt{\theta_c} \dim \CC} \sum_{k, l \in \Irr(\CC)} d_k \ol{S}_{c,l} N^{kl}_a \frac{\theta^2_k}{\theta^2_l}\right)\,,
\end{equation}
where $d_k$ is the categorical dimension of the object $k$
\end{thm}
\begin{proof}
For the last assertion, it follows from \cite[Prop. 6.1]{NS4} that
\begin{equation}\label{eq:2nd_ind_modular}
\nu_{2,1}^{\iota(c)}(a) = \frac{1}{\dim \CC} \sum_{k, l \in \Irr(\CC)} d_k \ol{S}_{c,l} N^{kl}_a \frac{\theta^2_k}{\theta^2_l},
\end{equation}
where $S_{c,l}$ and $T_{c,l}=\delta_{c,l}\theta_l$ are the modular data of $\CC$. By the Verlinde formula, the fusion rules of $\CC$ are determined by the $S$-matrix. Therefore, the braiding of $\CC$ is completely determined by the modular data $(S,T)$ of $\CC$. 
\end{proof}
\begin{remark}  
Let $\CC$ be a modular tensor category. It was conjectured in \cite{PSS} that
$$
Y_{a,b}^c:=\frac{\sqrt{\theta_b}}{\sqrt{\theta_{c}}\dim \CC} \sum_{k, l \in \Irr(\CC)} \ol{S}_{b,k} \ol{S}_{c,l} N^{kl}_a \frac{\theta^2_k}{\theta^2_l} \in \BZ 
$$ 
for  all $a, b, c \in \Irr(\CC)$. It was further conjectured in \cite{BHS} that 
$
\dim \CC(c, a\o a\o b) \pm Y_{a,b}^c \ge 0\,.
$
Both conjectures were proved in \cite[\S 8]{NS4}, and it was further proved in \cite[Prop. 4.2]{BCT} that $\dim \CC(c, a\o a\o b) \pm Y_{a,b}^c$ is always an even integer. The second statement of Theorem \ref{t:2} could be considered as an extension of \cite[Prop. 4.2]{BCT} to premodular categories but only for $b = \1$.

The expression \eqref{eq:2nd_indicator_R} also appeared in \cite[Prop. 2.3 (1)]{BPD} and \cite[Prop. 2.4 (1)]{GR} with $\nu_{2,1}^{\iota(c)}(a)$ being substituted by \eqref{eq:2nd_ind_modular} and the Verlinde formula.  Equation \eqref{eq2} can also be found in \cite{Bantay97}. A similar conclusion for UMTC as the last part of Theorem 3 was also obtained in \cite[Prop. 3.2]{GR}.
\end{remark}

\noindent
{\bf Acknowledgements:}

S.-H. N. was partially supported by NSF grant  DMS-1664418 and the Simons Foundations.
E.C.R. was partially supported by NSF grant DMS-2205962.  X.-G.W was partially
supported by NSF grant DMR-2022428 and by the Simons Collaboration on
Ultra-Quantum Matter, which is a grant from the Simons Foundation (651446,
XGW).  The authors have no relevant financial or non-financial interests to
disclose.

This material is based upon work supported by the National Science Foundation under Grant No. DMS-1928930, while the first two authors were in residence at the Mathematical Sciences Research Institute in Berkeley, California,  in 2024.\\

\noindent
{\bf Conflict of interest:} On behalf of all authors, the corresponding author states that there is no conflict of interest.\\

\noindent
{\bf Data Availability:} The study did not involve any data.
\bibliographystyle{unsrt}
%\bibliography{mybibl}

\end{document}